\documentclass[12pt]{amsart}
\usepackage{amssymb,verbatim,amscd,amsmath,graphicx}
\usepackage{graphicx}
\usepackage{caption}
\usepackage{subcaption}
\usepackage{pdfsync}

\numberwithin{equation}{section}
\newtheorem{theorem}{Theorem}[section]
\newtheorem{lemma}[theorem]{Lemma}

\newtheorem{corollary}[theorem]{Corollary}

\theoremstyle{definition}
\newtheorem{definition}[theorem]{Definition}
\newtheorem{def-prop}[theorem]{Definition-Proposition}

\newtheorem{remark}[theorem]{Remark}
\newtheorem{example}[theorem]{Example}

\newtheorem{question}[theorem]{Question}

\DeclareMathOperator{\reg}{reg}

\DeclareMathOperator{\supp}{supp}

\newcommand{\BR}{{\Big\rfloor}}
\newcommand{\BL}{{\Big\lfloor}}

\newcommand{\ZZ}{{\mathbb Z}}
\newcommand{\NN}{{\mathbb N}}

\def\alb {\mathbf {\alpha}}

\def\h {\widetilde{H}}

\def\1{{\bf 1}}
\def\0{{\bf 0}}

\begin{document}


\title{Regularity of powers of forests and cycles}

\author{Selvi Beyarslan}
\address{Tulane University \\ Department of Mathematics \\
6823 St. Charles Ave. \\ New Orleans, LA 70118, USA}
\email{sbeyarsl@tulane.edu}

\author{Huy T\`ai H\`a}
\address{Tulane University \\ Department of Mathematics \\
6823 St. Charles Ave. \\ New Orleans, LA 70118, USA}
\email{tha@tulane.edu}
\urladdr{http://www.math.tulane.edu/$\sim$tai/}

\author{Tr\^an Nam Trung}
\address{Institute of Mathematics \\ Vietnam Academy of Science and Technology, VAST \\ 18 Hoang Quoc Viet \\ Hanoi, Vietnam}
\email{tntrung@math.ac.vn}


\begin{abstract}
Let $G$ be a graph and let $I = I(G)$ be its edge ideal. In this paper, when $G$ is a forest or a cycle, we explicitly compute the regularity of $I^s$ for all $s \ge 1$. In particular, for these classes of graphs, we provide the asymptotic linear function $\reg(I^s)$ as $s \gg 0$, and the initial value of $s$ starting from which $\reg(I^s)$ attains its linear form. We also give new bounds on the regularity of $I$ when $G$ contains a Hamiltonian path and when $G$ is a Hamiltonian graph.
\end{abstract}

\maketitle


\section{Introduction} \label{sec.intro}

Let $R = k[x_1, \dots, x_n]$ be a polynomial ring and let $I \subset R$ be a homogeneous ideal. It is a celebrated result that $\reg(I^s)$ is asymptotically a linear function for $s \gg 0$ (cf. \cite{Chardin2007, CHT, Ko, TW}). However, the question of starting from which value of $s$ this function becomes linear and finding the exact form of the linear function continues to elude us (cf. \cite{Be, Ch, EH, EU, Ha1}), even in simple situations such as when $I$ is generated by general linear forms (cf. \cite{EH}) or when $I$ is a monomial ideal (cf. \cite{Conca, Hasurvey}). In this paper, we address this question when $I = I(G)$ is the edge ideal of a graph. In this case, there exist integers $b$ and $s_0$ such that $\reg(I^s) = 2s + b$ for all $s \ge s_0$, and the problem is to identify $b$ and $s_0$ via combinatorial structures of the graph $G$.

There are very few examples of graphs for which $b$ and $s_0$ are found. Even in the simplest case, it is a difficult problem to characterize graphs with $b=0$, i.e., when $\reg(I(G)^s) = 2s \ \forall \ s \gg 0$ (cf. \cite{NP}). Herzog, Hibi and Zheng \cite{HHZ} showed that if $I(G)$ has a linear resolution then so does $I(G)^s$ for all $s \ge 1$. This and Fr\"oberg's characterization of graphs whose edge ideals have linear resolutions (cf. \cite{Fr, EGHP, HVT2005}) imply that if the complement graph of $G$ contains no induced cycles of length $\ge 4$ then $b = 0$ and $s_0 = 1$. Ferr\`o, Murgia and Olteanu \cite{FMNO} showed that it is also the case that $b = 0$ and $s_0=1$ for initial and final lexsegment edge ideals. In a recent preprint, Alilooee and Banerjee \cite{AB} proved that if $G$ is a bipartite graph such that $\reg(I(G)) = 3$ then $\reg(I(G)^s) = 2s+1$ for all $s \ge 1$, i.e., $b = 1$ and $s_0 = 1$. Our results add to this list of rare instances a number of new classes of graphs for which $b$ and $s_0$ can be computed explicitly. More specifically, we shall compute $\reg(I(G)^s)$ for all $s \ge 1$ when $G$ is a \emph{forest} and when $G$ is a \emph{cycle}.

As in many previous works in the literature (cf. \cite{BHO, Fa, HVT2005, HVT2008, MO}), we start with the case when $G$ contains no cycles. That is, when $G$ is a forest. Our first result is stated as follows.

\begin{theorem}[Theorem \ref{T2}] \label{intro11}
Let $G$ be a forest with edge ideal $I = I(G)$. Let $\nu(G)$ denote the induced matching number of $G$. Then for all $s \ge 1$, we have
$$\reg(I^s) = 2s + \nu(G)-1.$$
\end{theorem}
\noindent Notice that for this class of graphs, we have $b = \nu(G) - 1$ is not zero in general, but $s_0$ is still 1.

Moving away from forests, the next class of graphs to consider are cycles. Our second result exhibits the first class of graphs for which $\reg(I(G)^s)$ can be computed for all $s$ and $s_0 \not= 1$.

\begin{theorem}[Theorem \ref{thm.cycle}] \label{intro12}
Let $C_n$ be the $n$-cycle and let $I = I(C_n)$ be its edge ideal. Let $\nu = \lfloor \frac{n}{3} \rfloor$ denote the induced matching number of $C_n$. Then
$$\reg(I) = \left\{ \begin{array}{rcll} \nu + 1 & \text{if} & n \equiv 0,1 & (\text{mod } 3) \\
\nu + 2 & \text{if} & n \equiv 2 & (\text{mod } 3), \end{array} \right.$$
and for all $s \ge 2$, we have
$$\reg(I^s) = 2s + \nu-1.$$
\end{theorem}

The computations in Theorems \ref{intro11} and \ref{intro12} are inspired by a simple bound for the regularity of the edge ideal of any graph, namely, $\reg(I(G)) \ge \nu(G) + 1$ (cf. \cite{HVT2008, K, MV}). We extend this bound to include all powers of $I(G)$. In fact, in Theorem \ref{T1}, we prove that if $G$ is any graph and $\nu(G)$ denotes its induced matching number then for all $s \ge 1$,
\begin{align}
\reg(I(G)^s) & \ge 2s + \nu(G)-1. \label{eq.intro}
\end{align}

In light of (\ref{eq.intro}) it then remains to establish the inequality in the other direction. To to so, in Theorem \ref{intro11} (i.e., when $G$ is a forest), we use a non-standard inductive technique. Our induction is based on a sum of different powers of edge ideals of induced subgraphs, noting that a subgraph of a forest is also a forest. Particularly, in examining the regularity of $I(G)^s$, we investigate the regularity of ideals of the form $I(G') + I(G'')^s$, where $G'$ and $G''$ are induced subgraphs of $G$ having no edges in common and whose union is $G$ itself. Observe that in the two extreme cases: (a) when $G'$ has no edges, this sum recovers $I(G)^s$; and (b) when $G''$ has no edges, this sum amounts to the edge ideal of a forest, something we understand quite well.

To settle a similar task in Theorem \ref{intro12} (i.e., when $G$ is a cycle), we make use of Banerjee's recent work \cite{Ba}, in which the regularity of $I(G)^s$ can be bounded via that of ideals of the form $I(G)^s : M$, where $M$ is a minimal generator of $I(G)^{s-1}$. These ideals in general are not squarefree. Using polarization we reduce these ideals to edge ideals of a different class of graphs. The problem is now to bound the regularity of edge ideals of those new graphs. We accomplish this by providing new bounds for the regularity of $I(G)$ when $G$ is a graph containing Hamiltonian paths or Hamiltonian cycles.

Our new bounds for the regularity of $I(G)$ when $G$ contains Hamiltonian paths or cycles are interesting on their own. Finding bounds for the regularity of $I(G)$ in terms of combinatorial data of $G$ is an active research program in combinatorial commutative algebra in recent years (see \cite{Hasurvey} and references therein). For any graph $G$, making use of the matching number bound of \cite{HVT2005}, it is easy to see that $\reg(I(G)) \le \lfloor \frac{n}{2} \rfloor + 1$. Our results show that this bound can be improved significantly when $G$ contains Hamiltonian paths or cycles. Our proof of Theorem \ref{intro2} is a combination of induction on the number of vertices and the degree of certain vertices in $G$.

\begin{theorem}[Theorems \ref{thm.Hamiltonpath} and \ref{thm.Hamiltoncycle}] \label{intro2}
Let $G$ be a graph over $n$ vertices.
\begin{enumerate}
\item If $G$ contains a Hamiltonian path then
$$\reg(I(G)) \le \Big\lfloor \dfrac{n+1}{3} \Big\rfloor + 1.$$
\item If $G$ contains a Hamiltonian cycle then
$$\reg(I(G)) \le \Big\lfloor \dfrac{n}{3} \Big\rfloor + 1.$$
\end{enumerate}
\end{theorem}

Our paper is structured as follows. In the next section, we collect notations and terminology used in the paper, and recall a few auxiliary results. In Section \ref{sec.Hamilton}, we provide new bounds for the regularity of an edge ideal when the graph contains a Hamiltonian path or it is a Hamiltonian graph. Sections \ref{sec.forest} and \ref{sec.cycle} are devoted to prove our main results, Theorems \ref{intro11} and \ref{intro12}. We start Section \ref{sec.forest} by establishing the bound (\ref{eq.intro}), and continue with the proof of Theorem \ref{intro11}. The proof of Theorem \ref{intro12} is given in Section \ref{sec.cycle}.

\vspace{1em}

\noindent{\bf Acknowledgement.} Part of this work was done while H\`a and Trung were at the Vietnam Institute of Advanced Studies in Mathematics (VIASM) in Hanoi, Vietnam. We would like to thank VIASM for its hospitality. H\`a is partially supported by the Simons Foundation (grant \#279786). Trung acknowledges supports partially from NAFOSTED (Vietnam), project 101.01-2011.48. We would also like to thank an anonymous referee for many helpful comments.


\section{Preliminaries} \label{sec.prel}

We shall follow standard notations and terminology from usual texts in the research area (cf. \cite{BrHe, E, HH, MS}).

Let $k$ be a field, let $R = k[x_1, \dots, x_n]$ be a standard graded polynomial ring over $n$ variables. The object of our work is the Castelnuovo-Mumford regularity of graded modules and ideals over $R$. This invariant can be defined in various ways. For our purpose, we recall the definition that uses the minimal free resolution.

\begin{definition} Let $T$ be a finitely generated graded $R$-module and let
$$0 \rightarrow \bigoplus_{j \in \ZZ} R(-j)^{\beta_{p,j}(T)} \rightarrow \cdots \rightarrow \bigoplus_{j \in \ZZ} R(-j)^{\beta_{0,j}(T)} \rightarrow T \rightarrow 0$$
be its minimal free resolution. Then the \emph{regularity} of $T$ is defined by
$$\reg(T) = \max \{j-i ~|~ \beta_{i,j}(T) \not= 0\}.$$
\end{definition}

We identify the variables of $R$ with distinct vertices $V = \{x_1, \dots, x_n\}$. A \emph{graph} $G = (V,E)$ consists of $V$ and a set $E$ of edges connecting pairs of vertices (we think of an edge as a set containing two vertices). When not already specified, we use $V_G$ and $E_G$ to denote the vertices and edges, respectively, of a given graph $G$. Throughout the paper, we shall restrict our attention to \emph{simple} graphs, i.e., graphs without loops nor multiple edges. The algebra-combinatorics framework used in this paper is described via the edge ideal construction.

\begin{definition} Let $G = (V,E)$ be a graph. The \emph{edge ideal} of $G$ is defined to be
$$I(G) = \big( uv ~|~ \{u,v\} \in E \big) \subseteq R = k[x_1, \dots, x_n].$$
\end{definition}

\begin{remark} When discussing the regularity of edge ideals, for simplicity of notation, we shall use $\reg(G)$ to also refer to $\reg(I(G))$.
\end{remark}

\begin{remark} We often write $uv \in E$ instead of $\{u,v\} \in E$. By abusing notation, we use $uv$ to refer to both the edge $uv \in E$ and the monomial $uv \in I(G)$. If $e = \{u,v\}$ is an edge then we shall further write $x^e$ to denote the monomial $uv$ corresponding to $e$.
\end{remark}

For a vertex $u$ in a graph $G = (V,E)$, let $N_G(u) = \{v \in V ~|~ uv \in E\}$ be the set of \emph{neighbors} of $u$, and set $N_G[u] := N_G(u) \cup \{u\}$. An edge $e$ is \emph{incident} to a vertex $u$ if $u \in e$. The \emph{degree} of a vertex $u \in V$, denoted by $\deg_G(u)$, is the number of edges incident to $u$. When there is no confusion, we shall omit $G$ and write $N(u), N[u]$ and $\deg(u)$.

For an edge $e$ in a graph $G$, define $G\setminus e$ to be the subgraph of $G$ with
the edge $e$ deleted (but its vertices remained). For a subset $W \subseteq V$ of the vertices in $G$, define $G \setminus W$ to be the subgraph of $G$ with the vertices in $W$ (and their incident edges) deleted. When $W = \{u\}$ consists of a single vertex, we write $G \setminus u$ instead of $G \setminus \{u\}$. If $e = \{u,v\}$ then set $N_G[e] = N_G[u] \cup N_G[v]$ and define $G_e$ to be the subgraph $G \setminus N_G[e]$ of $G$.

\begin{remark} \label{rmk.flat}
Let $R \hookrightarrow S$ be a ring extension obtained by adjoining new variables to $R$. Let $I \subseteq R$ be a homogeneous ideal and let $IS$ be its extension in $S$. Since the extension $R \hookrightarrow S$ is flat, the minimal free resolution of $IS$ (as an $S$-module) can be obtained from that of $I$ (as an $R$-module) by tensoring with $S$. Thus, $\reg(I) = \reg(IS)$. This allows us to write $\reg(I)$ for both the regularity of $I$ as an $R$-module and the regularity of $IS$ as an $S$-module. As a consequence, if $G$ is a graph containing an isolated vertex $u$ then $\reg(I(G)) = \reg(I(G \setminus u))$, and we can freely drop the vertex $u$ from $G$.
\end{remark}

\begin{remark} \label{rmk.1var}
Let $y$ be a new indeterminate and let $S = R[y]$. Let $I$ be a homogeneous ideal in $R$. Observe that $S\big/(I,y) \simeq R/I \otimes_k k[y]/(y)$. This implies that the minimal free resolution of $S/(I,y)$ is obtained by taking the tensor product of those of $R/I$ and $k[y]/(y)$. Thus, $\reg\big(S\big/(I,y)\big) = \reg(R/I)$, where the first regularity is taken over $S$ and the second regularity is taken over $R$. This and Remark \ref{rmk.flat} allow us to conclude without any ambiguity that
$\reg(I) = \reg(I + (y)).$
\end{remark}

A graph $H$ is called an \emph{induced subgraph} of $G$ if the vertices of $H$ are vertices in $G$, and for vertices $u$ and $v$ in $H$, $\{u,v\}$ is an edge in $H$ if and only if $\{u,v\}$ is an edge in $G$. The induced subgraph of $G$ over a subset $W \subseteq V$ is obtained by deleting vertices not in $W$ from $G$ (and their incident edges).

In the study of the regularity of edge ideals, induction have proved to be a powerful technique (see \cite{Hasurvey}). We recall a number of inductive results stated in \cite{Hasurvey} that we shall use in the paper.

\begin{theorem}[\protect{See \cite[Lemma 3.1, Theorems 3.4 and 3.5]{Hasurvey}}] \label{thm.induction}
Let $G = (V,E)$ be a graph.
\begin{enumerate}
\item If $H$ is an induced subgraph of $G$ then
$\reg(H) \le \reg(G)$.
\item Let $x \in V$. Then
$$\reg(G) \le \max \{\reg(G \setminus x), \reg(G \setminus N[x]) + 1\}.$$
\item Let $e \in E$. Then
$$\reg(G) \le \max \{2, \reg(G \setminus e), \reg(G_e) + 1\}.$$
\end{enumerate}
\end{theorem}

In finding combinatorial bounds for the regularity of edge ideals, matching numbers and induced matching numbers have been used often. We recall their definitions.

\begin{definition} Let $G = (V,E)$ be a graph.
\begin{enumerate}
\item A collection of edges $\{e_1, \dots, e_s\} \subseteq E$ is called a \emph{matching} if they are pairwise disjoint. The largest size of a matching in $G$ is called its \emph{matching number} and denoted by $\beta(G)$.
\item A collection of edges $\{e_1, \dots, e_s\} \subseteq E$ is called an \emph{induced matching} if they form a matching and they are the only edges in the induced subgraph of $G$ over the vertices $\bigcup_{i=1}^s e_i$. The largest size of an induced matching in $G$ is called its \emph{induced matching number} and denoted by $\nu(G)$.
\end{enumerate}
\end{definition}

\begin{example}
Let $G$ be the graph in Figure \ref{fig.example}. Then $\{x_1x_2, x_3x_4\}$ forms a matching, but not an induced matching (the induced subgraph on $\{x_1, \dots, x_4\}$ also contains edges $\{x_2x_3, x_1x_4\}$). The set $\{x_1x_2, x_3x_4, x_5x_6\}$ forms a maximal matching of size 3. It is not hard to verify that the matching number $\beta(G)$ is 3. On the other hand, the induced matching number $\nu(G)$ is 1.
\begin{figure}[h!]
\centering
\includegraphics[height=2in]{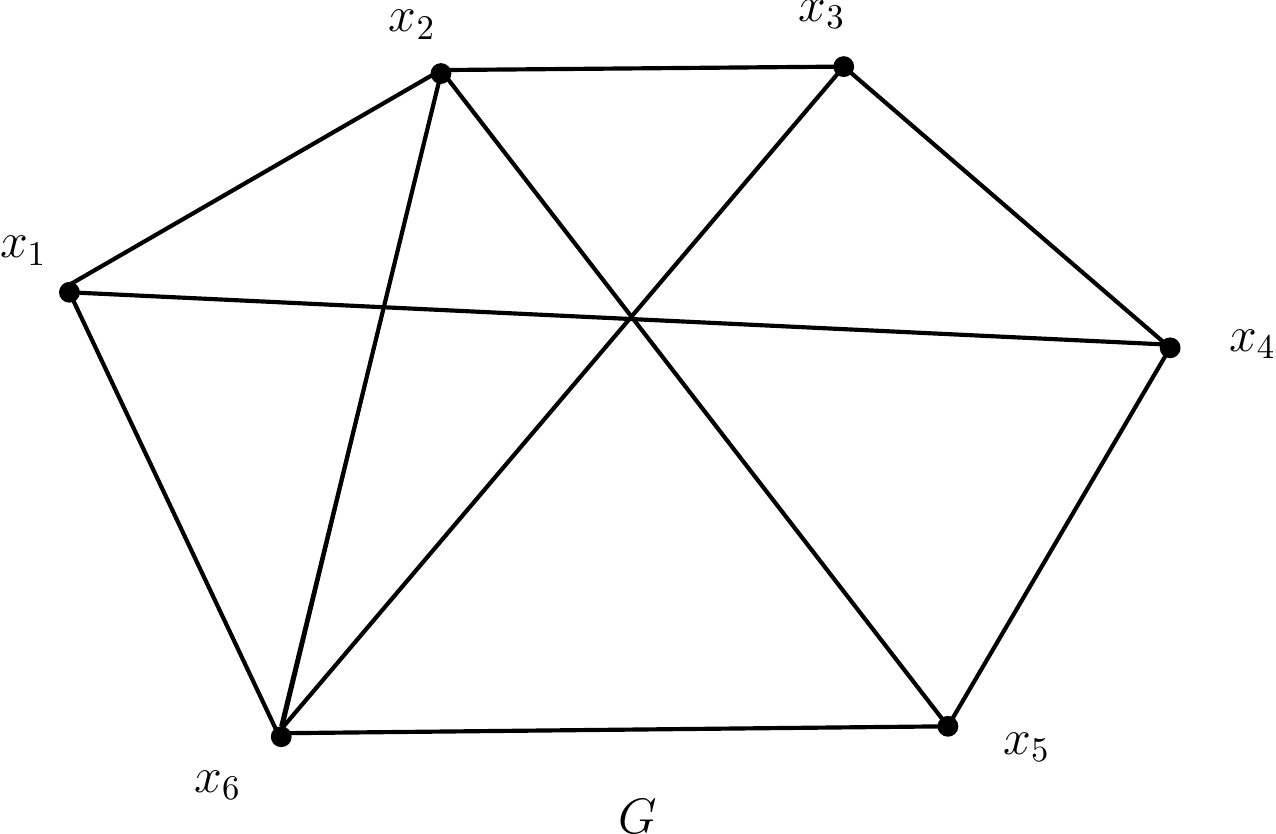}
\caption{A running example.}\label{fig.example}
\end{figure}
\end{example}

\begin{definition} Let $G = (V,E)$ be a graph.
\begin{enumerate}
\item A \emph{walk} in $G$ is an alternating sequence of vertices and edges
$$u_1, e_1, u_2, e_2, \dots, e_{m-1}, u_m,$$
in which $e_i = \{u_i, u_{i+1}\}$. In a simple graph, specifying an edge is the same as specifying its vertices. Thus, we shall omit the edges in $u_1, e_1, \dots, e_{m-1}, u_m$ and use $u_1, \dots, u_m$ to denote such a walk.
\item A path is a \emph{simple} walk; that is, a walk in which each vertex appears exactly once (except possibly the first and the last vertices).
\item A \emph{circuit} is a closed walk, i.e., a walk $u_1, \dots, u_m$, where $u_1 \equiv u_m$.
\item A \emph{cycle} is a closed path. A cycle consisting of $n$ distinct vertices is called an $n$-cycle. When listing the vertices of a cycle, the last vertex (which is the same as the first one) is often omitted from the sequence.
\item A \emph{Hamiltonian path} of $G$ is a path that goes through each vertex of $G$ exactly once.
\item A \emph{Hamiltonian cycle} of $G$ is a cycle that contains all the vertices of $G$ (each appears exactly once). A graph containing a Hamiltonian cycle is called a \emph{Hamiltonian graph}.
\end{enumerate}
\end{definition}

\begin{example} Let $G$ be the graph in Figure \ref{fig.example}. Then $G$ is a Hamiltonian graph; it admits $x_1, \dots, x_6$ as a Hamiltonian cycle. On the other hand, $x_1, x_4, x_5, x_2, x_6, x_3$ forms a Hamiltonian path, but is not a Hamiltonian cycle.
\end{example}

\begin{remark} Let $G$ be a path with vertices $x_1, \dots, x_n$ (in order), where $x_1 \not\equiv x_n$. Then it follows straightforward from the definition that an induced matching with maximal size is given by $\{x_1, x_2\}, \{x_4,x_5\}, \dots$. Thus, $\nu(G) = \lfloor \frac{n+1}{3} \rfloor.$
\end{remark}

\begin{remark} Let $G$ be the $n$-cycle with vertices $x_1, \dots, x_n$ (in order). Observe that if $e = \{x_1,x_2\}$ is an edge in an induced matching $\mathcal{C}$, then edges in $\mathcal{C} \setminus \{e\}$ cannot contain $x_3$ nor $x_n$. Thus, $\mathcal{C} \setminus \{e\}$ is an induced matching of the $(n-4)$-path $x_4, \dots, x_{n-1}$. It then follows that
$$\nu(G) = 1 + \BL \dfrac{n-3}{3} \BR = \BL \dfrac{n}{3} \BR.$$
\end{remark}

Our method of proving one of the main results, Theorem \ref{intro12}, is making use of Banerjee's recent work \cite{Ba}. We recall the following definition and theorem from \cite{Ba}.

\begin{definition} \label{def.evenconnected}
Let $G = (V,E)$ be a graph with edge ideal $I = I(G)$. Two vertices $u$ and $v$ in $G$ are said to be \emph{even-connected} with respect to an $s$-fold product $M = x^{e_1} \dots x^{e_s}$, where $e_1, \dots, e_s$ are edges in $G$, if there is a path $p_0, \dots, p_{2l+1}$, for some $l \ge 1$, in $G$ such that the following conditions hold:
\begin{enumerate}
\item $p_0 \equiv u$ and $p_{2l+1} \equiv v$;
\item for all $0 \le j \le l-1$, $\{p_{2j+1}, p_{2j+2}\} = e_i$ for some $i$; and
\item for all $i$,
$\big| \{j ~|~ \{p_{2j+1}, p_{2j+2}\} = e_i\} \big| \le \big| \{t ~|~ e_t = e_i\} \big|.$
\end{enumerate}
\end{definition}

\begin{theorem}[\protect{\cite[Theorems 6.1 and 6.7]{Ba}}] \label{thm.Ba}
Let $G = (V,E)$ be a graph with edge ideal $I = I(G)$, and let $s \ge 1$ be an integer. Let $M = x^{e_1} \dots x^{e_s}$ be a minimal generator of $I^s$. Then $(I^{s+1}: M)$ is minimally generated by monomials of degree 2, and $uv$ ($u$ and $v$ may be the same) is a minimal generator of $(I^{s+1}:M)$ if and only if either $\{u,v\} \in E$ or $u$ and $v$ are even-connected with respect to $M$.
\end{theorem}

\begin{example} Let $G$ be the graph in Figure \ref{fig.example} and let $I = I(G) \subseteq k[x_1, \dots, x_6]$. Let $e = \{x_2, x_6\}$. Then $x_3$ and $x_5$ is even-connected with respect to $M = x^e = x_2x_6$. The path $p_0, \dots, p_3$ as in Definition \ref{def.evenconnected} can be chosen to be $x_3,x_2,x_6,x_5$. In particular, $x_3x_5 \in (I^2 : M)$ by Theorem \ref{thm.Ba}.
\end{example}


\section{Graphs with Hamiltonian paths and Hamiltonian cycles} \label{sec.Hamilton}

In this section, we provide new bounds for the regularity of the edge ideal of a graph that contains a Hamiltonian path or a Hamiltonian cycle. These bounds are interesting on their own, but they will also be used later in the proofs of our main results.

\begin{theorem} \label{thm.Hamiltonpath}
Let $G$ be a graph on $n$ vertices. Assume that $G$ contains a Hamiltonian path. Then
$$\reg(G) \le \Big\lfloor \dfrac{n+1}{3} \Big\rfloor + 1.$$
\end{theorem}

\begin{proof} We use induction on the number of vertices in $G$. The result can be verified trivially for $n \le 2$. Assume now that $n > 2$. Without loss of generality, suppose that $x_1,x_2, \dots, x_n$ forms a Hamiltonian path in $G$. Let $d = \deg_G(x_1)$ be the degree of $x_1$ in $G$. Our proof proceeds by induction on $d$.

Consider the case when $d = 1$, i.e., $x_1$ is a leaf of $G$. Let $H = G \setminus \{x_1, x_2\}$. It can be seen that $x_3, \dots, x_n$ is a Hamiltonian path in $H$. Also, $x_1$ is an isolated vertex in $G \setminus x_2$. Thus, $\reg(G \setminus x_2) = \reg(H)$. Therefore, by the induction hypothesis on the number of vertices, we get
\begin{align}
\reg(G \setminus x_2) = \reg(H) & \le \Big\lfloor \dfrac{(n-2) + 1}{3} \Big\rfloor + 1 \le \Big\lfloor \dfrac{n+1}{3} \Big\rfloor + 1. \label{eq.d11}
\end{align}

Let $K$ be the induced subgraph of $G$ over the vertices $\{x_4, \dots, x_n\}$. It is easy to see that $G \setminus N[x_2]$ is an induced subgraph of $K$. It follows from Theorem \ref{thm.induction} that $\reg(G \setminus N[x_2]) \le \reg(K)$. Moreover, $K$ has a Hamiltonian path $x_4, \dots, x_n$. Thus, by induction on $n$, we get
\begin{align}
\reg(K) & \le \BL \dfrac{(n-3)+1}{3} \BR + 1 = \BL \dfrac{n+1}{3} \BR. \label{eq.d12}
\end{align}
It then follows from Theorem \ref{thm.induction}, together with (\ref{eq.d11}) and (\ref{eq.d12}), that
$$\reg(G) \le \max\{\reg(G \setminus x_2), \reg(G \setminus N[x_2]) + 1\} \le \BL \dfrac{n+1}{3} \BR + 1.$$

We may suppose now that $d \ge 2$, i.e., $x_1$ is connected to some $x_i$s with $i > 2$. Let $3 \le t \le n$ be the smallest integer which that $\{x_1, x_t\}$ is an edge in $G$. 

\noindent{\it Case 1:} $t = n$. In this case, $\deg_G(x_1) = 2$ and $G$ is a Hamiltonian graph. Observe that $G \setminus x_1$ contains a Hamiltonian path $x_2, \dots, x_n$, and so, by induction on $n$, we have
$$\reg(G \setminus x_1) \le \BL \dfrac{(n-1) + 1}{3} \BR + 1 \le \BL \dfrac{n+1}{3} \BR + 1.$$

Let $K$ be the induced subgraph of $G$ over the vertices $\{x_3, \dots, x_{n-1}\}$. Observe further that $K$ contains a Hamiltonian path $x_3, \dots, x_{n-1}$ and $G \setminus N[x_1]$ is an induced subgraph of $K$. Thus, by induction on $n$ and Theorem \ref{thm.induction}, we have
$$\reg(G \setminus N[x_1]) \le \reg(K) \le \BL \dfrac{(n-3)+1}{3} \BR + 1 = \BL \dfrac{n+1}{3} \BR.$$
It follows again from Theorem \ref{thm.induction} that
$$\reg(G) \le \max\{\reg(G \setminus x_1), \reg(G \setminus N[x_1]) + 1\} \le \BL \dfrac{n+1}{3} \BR + 1.$$

\noindent{\it Case 2:} $3 \le t < n$. Let $Q$ be the subgraph obtained from $G$ by removing the edge $e=\{x_1, x_t\}$. Then $Q$ has a Hamiltonian path $x_1, \dots, x_n$ and $\deg_Q(x_1) < \deg_G(x_1) = d$. Thus, by the induction hypothesis on $d$, we get
\begin{align}
\reg(Q) & \le \BL \dfrac{n+1}{3} \BR + 1. \label{eq.Q1}
\end{align}

Recall that $G_e$ is the subgraph obtained from $G$ by removing the vertices in $N[e]$. It follows from Theorem \ref{thm.induction} that
$$\reg(G) \le \max\{\reg(Q), \reg(G_e)+1\}.$$
Thus, in light of (\ref{eq.Q1}), to prove that $\reg(G) \le \BL \dfrac{n+1}{3} \BR + 1$, it remains to show that
\begin{align}
\reg(G_e) & \le \BL \dfrac{n+1}{3} \BR. \label{eq.Q2}
\end{align}

Indeed, if $3 \le t \le 4$ then $G_e$ is an induced subgraph of $H$, where $H$ is the induced subgraph of $G$ over the vertices $\{x_{t+1}, \dots, x_n\}$, which has a Hamiltonian path; and thus, by induction on the number of vertices and Theorem \ref{thm.induction}, we have
$$\reg(G_e) \le \reg(H) \le \BL \dfrac{(n-t)+1}{3} \BR + 1 \le \BL \dfrac{n+1}{3} \BR,$$
and (\ref{eq.Q2}) holds.

\begin{figure}[h!]
\centering
\includegraphics[height=2in]{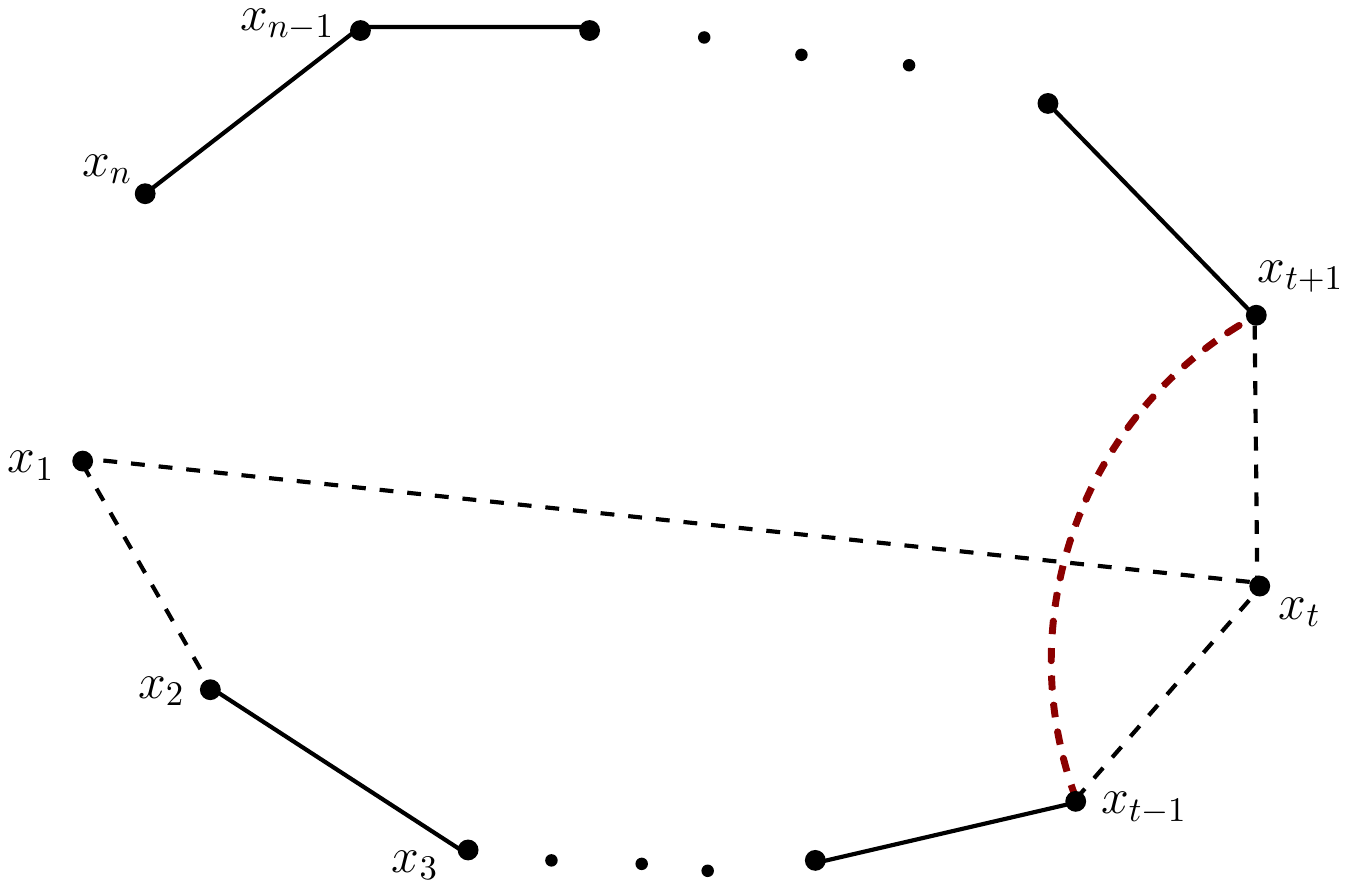}
\caption{Construction of $H'$ in the proof of Theorem \ref{thm.Hamiltonpath}.}\label{fig.path}
\end{figure}

If, on the other hand, $t \ge 5$ then let $H$ be the induced subgraph of $G$ over the vertices $V_G \setminus \{x_1, x_2, x_t\}$, and let $H'$ be the graph obtained from $H$ by connecting $x_{t-1}$ and $x_{t+1}$ if this edge is not already in $G$ (see Figure \ref{fig.path}). Clearly, $H'$ has a Hamiltonian path $x_3, \dots, x_{t-1}, x_{t+1}, \dots, x_n$. Thus, by induction on the number of vertices, we get
$$\reg(H') \le \BL \dfrac{(n-3)+1}{3} \BR + 1 = \BL \dfrac{n+1}{3} \BR.$$
Moreover, it can be seen that $G_e$ is an induced subgraph of $H'$ (since $x_{t-1}$ and $x_{t+1}$ are in $N[e]$, they are not in $G_e$), which implies that $\reg(G_e) \le \reg(H')$, and (\ref{eq.Q2}) follows.
\end{proof}

When $G$ is a Hamiltonian graph, i.e., $G$ contains not just a Hamiltonian path but a Hamiltonian cycle, then we can slightly strengthen the bound in Theorem \ref{thm.Hamiltonpath}.

\begin{theorem} \label{thm.Hamiltoncycle}
Let $G$ be a Hamiltonian graph on $n$ vertices. Assume that $x_1, \dots, x_n, x_1$ forms a Hamiltonian cycle in $G$, and there exists a value $t$ so that $\{x_{t-1}, x_{t+2}\}$ is an edge in $G$. Then
$$\reg(G) \le \BL \dfrac{n}{3} \BR + 1.$$
\end{theorem}

\begin{proof} It is clear from the assumption that $|E| \ge n+1$. We shall use induction on $|E|$.

Consider the base case when $|E| = n+1$. That is, $G$ consists of a cycle $x_1, \dots, x_n, x_1$ and exactly one edge $\{x_{t-1}, x_{t+2}\}$. In this case, $G \setminus x_{t+2}$ is a path of $(n-1)$ vertices. Thus, it follows from \cite[Corollary 5.4]{BHO} (see also \cite[Theorem 7.7.34]{J}) that
$$\reg(G \setminus x_{t+2}) = \BL \dfrac{(n-1)+1}{3} \BR + 1 = \BL \dfrac{n}{3} \BR + 1.$$
On the other hand, $G \setminus N[x_{t+2}]$ consists of an isolated vertex $x_t$ and a path of $(n-5)$ vertices. Therefore, it also follows from \cite[Corollary 5.4]{BHO} (and \cite[Theorem 7.7.34]{J}) that
$$\reg(G \setminus N[x_{t+2}]) = \BL \dfrac{(n-5) + 1}{3} \BR + 1 = \BL \dfrac{n-1}{3} \BR \le \BL \dfrac{n}{3} \BR.$$
Hence, by Theorem \ref{thm.induction}, we have
$$\reg(G) \le \BL \dfrac{n}{3} \BR + 1.$$

Suppose now that $|E| > n+1$. Then the cycle $x_1, \dots, x_n, x_1$ in $G$ contains another chord beside $x_{t-1}x_{t+2}$. We are in the situation where $G$ has a Hamiltonian cycle with at least two chords. By re-indexing the vertices of $G$ and the given chord $x_{t-1}x_{t+2}$ if necessary (with a different value of $t$), we may assume without loss of generality that $e=x_1x_l$ is another chord that is different from $x_{t-1}x_{t+2}$. Let $Q = G \setminus e$ be obtained by deleting the edge $e$ from $G$. Clearly, $Q$ contains a Hamiltonian cycle $x_1, \dots, x_n, x_1$ and the chord $x_{t-1}x_{t+2}$. Thus, by induction on the number of edges, we get
$$\reg(Q) \le \BL \dfrac{n}{3} \BR + 1.$$

\begin{figure}[h!]
\centering
\includegraphics[height=2in]{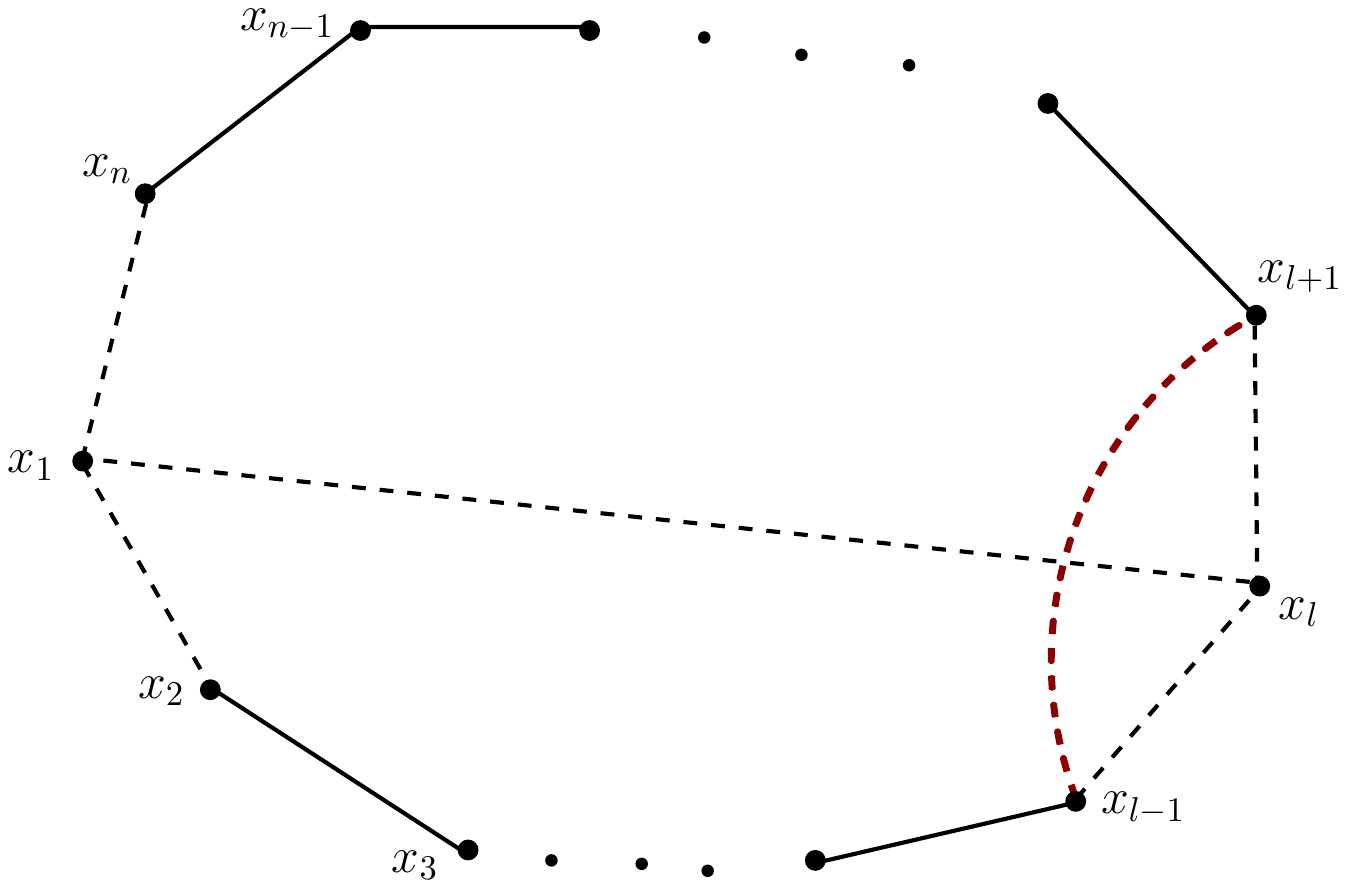}
\caption{Construction of $H'$ in the proof of Theorem \ref{thm.Hamiltoncycle}.}\label{fig.cycle}
\end{figure}

Let $H$ be the induced subgraph of $G$ over the vertices $V_G \setminus \{x_1, x_2, x_l, x_n\}$, and let $H'$ be the graph obtained from $H$ by connecting $x_{l-1}$ and $x_{l+1}$ if this edge is not already in $G$ (see Figure \ref{fig.cycle}). Then $H'$ has a Hamiltonian path $x_3, \dots, x_{l-1}, x_{l+1}, \dots, x_{n-1}$. Thus, by Theorem \ref{thm.Hamiltonpath}, we get
$$\reg(H') \le \BL \dfrac{(n-4) + 1}{3} \BR + 1 = \BL \dfrac{n}{3} \BR.$$
Moreover, $G_e$ is an induced subgraph of $H'$ (since $x_{l-1}$ and $x_{l+1}$ are in $N[e]$, they are not in $G_e$), so it follows from Theorem \ref{thm.induction} that
$$\reg(G_e) \le \reg(H') \le \BL \dfrac{n}{3} \BR.$$
The conclusion that $\reg(G) \le \lfloor \frac{n}{3} \rfloor + 1$ now also follows from Theorem \ref{thm.induction}, and the assertion is proved.
\end{proof}


\section{Regularity of powers of forests} \label{sec.forest}

In this section, we prove our first main result in explicitly computing the regularity of $I(G)^s$, for all $s \ge 1$, when $G$ is a forest. Our result, Theorem \ref{T2}, shows that in this case, $\reg(I(G)^s)$ is a linear function starting from $s_0 = 1$ and the constant $b$ is given by the induced matching number of $G$ minus 1. To accomplish this, we shall first establish the general bound for $\reg(I(G)^s)$, where $G$ is any graph, as stated in (\ref{eq.intro}).

We start by recalling the notion of upper-Koszul simplicial complexes associated to monomial ideals, whose reduced homology groups can be used to compute graded Betti numbers of these ideals.

\begin{definition} Let $I \subseteq R = k[x_1, \dots, x_n]$ be a monomial ideal and let $\alb=(\alpha_1,\ldots,\alpha_n) \in \NN^n$ be a $\NN^n$-graded degree. The \emph{upper-Koszul simplicial complex} associated to $I$ at degree $\alb$, denoted by $K^{\alb}(I)$, is the simplicial complex over $V = \{x_1, \dots, x_n\}$ whose faces are:
$$\Big\{W \subseteq V ~\Big|~ \dfrac{x^\alb}{\prod\limits_{u \in W} u} \in I\Big\}.$$
\end{definition}

Given a monomial ideal $I \subseteq R$, its $\NN^n$-graded Betti numbers are given by the following formula of Hochster (see \cite[Theorem 1.34]{MS}):
\begin{equation}\label{EQ1}
\beta_{i,\alb}(I) = \dim_k \h_{i-1}(K^{\alb}(I);k) \text{ for } i \ge 0 \text{ and } \alb \in \NN^n.
\end{equation}

Using (\ref{EQ1}) we obtain the following lemma, which is an analogue of \cite[Proposition 4.1.1]{J} for Betti numbers of powers of edge ideals.

\begin{lemma}\label{L1}
Let $G$ be a graph and let $H$ be an induced subgraph of $G$. Then for any $s \ge 1$ and any $i,j \ge 0$, we have
$$\beta_{i,j}(I(H)^s) \le \beta_{i,j}(I(G)^s).$$
\end{lemma}

\begin{proof} For an $\NN^n$-graded degree $\alb=(\alpha_1,\dots,\alpha_n)$, let $\supp(\alb) = \{x_i ~|~ \alpha_i \not= 0\}$ be the support of $\alb$. Observe that since $H$ is an induced subgraph of $G$, if $\supp(\alb) \subseteq V_H$ then $K^{\alb}(I(H)^s) = K^{\alb}(I(G)^s)$. Thus, it follows from $(\ref{EQ1})$ that
$$\beta_{i,\alb}(I(H)) = \dim_k \h_{i-1}(K^{\alb}(I(H)^s); k) = \dim_k \h_{i-1}(K^{\alb}(I(G)^s); k) = \beta_{i,\alb}(I(G)^s).$$
Hence,
\begin{align*} \beta_{i,j}(I(H)^s) &=\sum_{\alb\in\NN^n, \ \supp(\alb)\subseteq V_H, |\alb| = j} \beta_{i,\alb}(I(H)^s)=\sum_{\alb\in\NN^n, \ \supp(\alb)\subseteq V_H, |\alb| = j} \beta_{i,\alb}(I(G)^s)\\
&\le \sum_{\alb\in\NN^n, \ |\alb| = j} \beta_{i,\alb}(I(G)^s) = \beta_{i,j}(I(G)^s).
\end{align*}
\end{proof}

\begin{corollary}\label{C1} Let $G$ be a graph and let $H$ be an induced subgraph of $G$. Then, for all $s \ge 1$,
$$\reg I(H)^s \le \reg I(G)^s.$$
\end{corollary}

The following lemma is probably well-known. We include the proof for completeness.

\begin{lemma}\label{L2}
Let $F_1,\dots,F_r$ be a regular sequence of homogeneous polynomials in $R$ with $\deg F_1=\cdots=\deg F_r = d$. Let $I=(F_1,\dots,F_r)$. Then for all $s \ge 1$, we have
$$\reg(I^s) = ds + (d-1)(r-1).$$
\end{lemma}

\begin{proof} We use induction on $r$. The statement is clear for $r = 1$. Suppose that $r > 1$. We proceed by induction on $s$.
The statement is also clear if $s = 1$ by the Koszul complex. Thus, we may assume that $s \ge 2$.

Let $J = (F_1,\ldots,F_{r-1})$. Consider the following homomorphism
$$\phi: I^{s-1}(-d)\oplus J^s \stackrel{(F_r,1)}{\longrightarrow} I^s.$$
Since $I^s = (J + (F_r))^s = J^s + F_rI^{s-1}$, $\phi$ is surjective. Moreover, since $F_r$ is regular in $R/J$, the kernel of $\phi$ is given by $F_rJ^s$. Thus, we have the following short exact sequence
\begin{equation}\label{EQ2}
0 \longrightarrow J^s(-d)\longrightarrow I^{s-1}(-d)\oplus J^s \stackrel{(F_r,1)}{\longrightarrow} I^s\longrightarrow 0.
\end{equation}

By the induction hypothesis on $r$, we have
$\reg(J^s(-d))=\reg(J^s)+d = ds +(d-1)(r-1)+1.$
Furthermore, by the induction hypothesis on $s$, we have
$\reg(I^{s-1}(-d))= \reg(I^{s-1})+d = ds+(d-1)(r-1).$
Thus,
$$\reg(I^{s-1}(-d)\oplus J^s) = ds+(d-1)(r-1).$$
Combining with $(\ref{EQ2})$ and \cite[Corollary 20.19]{E}, we can conclude that
$$\reg(I^s) = ds + (d-1)(r-1),$$
and the lemma is proved.
\end{proof}

We are now ready to establish the general bound for $\reg(I(G)^s)$ stated in (\ref{eq.intro}).

\begin{theorem} \label{T1}
Let $G$ be a graph with edge ideal $I = I(G)$, and let $\nu(G)$ denote its induced matching number. Then for all $s \ge 1$, we have
$$\reg(I^s) \ge 2s+\nu(G)-1.$$
\end{theorem}

\begin{proof} For simplicity of notation, let $r = \nu(G)$. Suppose that $\{u_1v_1,\dots,u_rv_r\}$ is an induced matching in $G$. Let $H$ be the induced subgraph of $G$ on the vertices $\bigcup_{i=1}^r \{u_i,v_i\}$. Then, $I(H) = (u_1v_1,\dots,u_rv_r)$ is a complete intersection. Thus, by Lemma $\ref{L2}$, we have
$$\reg(I(H)^s) = 2s + (2-1)(r-1) = 2s + \nu(G)-1.$$
It now follows from Corollary $\ref{C1}$ that
$$\reg(I(G)^s) \ge \reg(I(H)^s) \ge 2s + \nu(G)-1.$$
\end{proof}

Recall that if $G$ and $H$ are two graphs then the \emph{union} of $G$ and $H$, denoted by $G \cup H$, is the graph with vertex set $V_G \cup V_H$ and edge set $E_G \cup E_H$. When $H$ consists of exactly one edge $e = \{u,v\}$, we write $G + e$ instead of $G \cup H$.

The following lemma is crucial in proving the reverse inequality of (\ref{eq.intro}) to get the statement of Theorem \ref{intro11}.

\begin{lemma}\label{L4}
Let $K$ be a forest and let $\nu(K)$ be its induced matching number. Suppose that $G$ and $H$ are induced subgraphs of $K$ such that
$$E_H \cup E_G = E_K \text{ and } E_H \cap E_G = \emptyset.$$
Then, for all $s \ge 1$, we have
$$\reg(I(H)+I(G)^s) \le 2s + \nu(K)-1.$$
\end{lemma}

\begin{proof} We shall use induction on $m:= s+|V_G|$. If $m = 1$ then we must have $s = 1$ and $V_G = \emptyset$. In this case, $E_H = E_K$, and so $I(H)+I(G)^s = I(K)$. The statement, thus, follows from \cite[Theorem 2.18]{Z}.

Suppose that $m \ge 2$. If $G$ consists of no edges then $E_H = E_K$, and we have $I(H) + I(G)^s = I(K)$. The assertion again follows from \cite[Theorem 2.18]{Z}.

Assume now that $E_G \not= \emptyset$. Being a subgraph of $K$, $G$ is a forest. In particular, $G$ contains a leaf. Let $x$ be a leaf in $G$ and let $y$ be the unique neighbor of $x$ in $G$. By \cite[Lemma 2.10]{MO}, we have $I(G)^s: xy = I(G)^{s-1}$. Since all ideals being discussed are monomial ideals, this implies that
$$(I(H)+I(G)^s):xy = (I(H):xy) + (I(G)^s:xy) = (I(H):xy) + I(G)^{s-1}.$$
Moreover,
$$(xy)+I(H)+I(G)^s = (xy)+I(H)+I(G\setminus x)^s = I(H+xy)+I(G\setminus x)^s.$$
Therefore, we have the following short exact sequence:
$$0 \rightarrow \big(R\big/I(H):xy+I(G)^{s-1}\big)(-2)\rightarrow R\big/(I(H)+I(G)^s)\rightarrow R/I(H+xy)+I(G\setminus x)^s\rightarrow 0.$$
This yields
\begin{equation}\label{EQ22}
\reg(I(H)+I(G)^s) \le \max\{\reg (I(H):xy+I(G)^{s-1})+2, \reg (I(H+xy)+I(G\setminus x)^s)\}.
\end{equation}

Let $\{u_1,\dots,u_p\} = N_H(x) \cup N_H(y)$ be all the vertices of $H$ which are adjacent to either $x$ or $y$. Let $H'= H \setminus \{u_1,\dots, u_p\}$. Then,
$$I(H):xy = I(H')+(u_1,\dots,u_p).$$
Observe that since $E_G \cap E_H = \emptyset$, none of the vertices $\{u_1,\dots,u_p\}$ are in $G$. Therefore, by Remark \ref{rmk.1var}, we have
$$\reg(I(H):xy+I(G)^{s-1}) = \reg(I(H')+(u_1,\dots,u_p)+I(G)^{s-1})= \reg(I(H')+I(G)^{s-1}).$$
This, coupled with (\ref{EQ22}), implies that
\begin{equation}\label{EQ23}
\reg(I(H)+I(G)^s) \le \max\{\reg(I(H')+I(G)^{s-1})+2, \reg(I(H+xy)+I(G\setminus x)^s)\}.
\end{equation}

Let $K':= H' \cup G$. Since $E_H \cap E_G=\emptyset$, we have $E_{H'}\cap E_G =\emptyset$. This implies that $K'$ is an induced subgraph of $K$. Thus, $K'$ is a forest, and
$$\nu(K') \le \nu(K).$$
Now, applying the induction hypothesis to $K', G$ and $H'$ with power $(s-1)$, we have
\begin{align}
\reg (I(H')+I(G)^{s-1}) & \le 2(s-1)+\nu(K')-1 \le 2(s-1) + \nu(K)-1. \label{eq.EQ1}
\end{align}
On the other hand, since $x$ is a leaf of $G$, we have $E_{H+xy} \cap E_{G\setminus x}  =\emptyset$ and $K = (H+xy) \cup (G\setminus x)$. Thus, we can apply the induction hypothesis to $K, G \setminus x$ and $H+xy$ to get
\begin{align}
\reg(I(H+xy)+I(G\setminus x)^s) & \le 2s+\nu(K)-1. \label{eq.EQ2}
\end{align}

Putting (\ref{EQ23}), (\ref{eq.EQ1}) and (\ref{eq.EQ2}) together we get the desired inequality
$$\reg (I(H)+I(G)^s) \le 2s + \nu(K)-1,$$
and the lemma is proved.
\end{proof}

We end this section by stating and proving our first main result.

\begin{theorem}\label{T2}
Let $G$ be a forest with edge ideal $I=I(G)$. Let $\nu(G)$ denote the induced matching number of $G$. Then for all $s \ge 1$, we have
$$\reg(I^s) = 2s + \nu(G)-1.$$
\end{theorem}

\begin{proof} For any $s\ge 1$, by Theorem \ref{T1} we have $\reg(I^s) \ge 2s + \nu(G)-1$. On the other hand, applying Lemma \ref{L4} by taking $K = G$ and $H = \emptyset$, we get $\reg(I^s) \le 2s + \nu(G)-1$.
Hence, for all $s \ge 1$,
$$\reg(I^s) = 2s + \nu(G)-1.$$
\end{proof}


\section{Regularity of powers of cycles} \label{sec.cycle}

In this section, we prove our second main result in explicitly computing the regularity of $I(G)^s$, for all $s \ge 1$, when $G$ is a cycle. For this class of graphs, $b$ can be computed from the induced matching number of $G$, but it is not necessarily true that $s_0 = 1$ as in the case for forests. Our method is to make use of Banerjee's recent work \cite{Ba} to reduce the problem to estimating the regularity of edge ideals of certain class of graphs; noting that these graphs contain Hamiltonian paths and cycles. To this end, we shall employ our new bounds for the regularity of edge ideals found in Section \ref{sec.Hamilton}.

We start with a lemma that specifies Theorem \ref{thm.Ba} to our situation, i.e., when $G$ is a cycle.

\begin{lemma} \label{lem.square}
Let $C_n$ be the $n$-cycle and assume that its vertices (in order) are $x_1, \dots, x_n$. Let $I = I(C_n)$. Then
$$x_n^2 \in (I^{s+1} : M),$$
where $M$ is a minimal generator of $I^s$, if and only if $n$ is odd, say $n=2l+1$ for some $1 \le l \le s$, and
$$M = (x_1x_2) \dots (x_{2l-1}x_{2l})N \text{ with } N \in I^{s-l}.$$
In this case we also have $x_nx_j \in I^{s+1} : M$ for all $j = 1, \dots, n$.
\end{lemma}

\begin{proof} Let us start by proving the ``if'' direction. Suppose that $n = 2l+1$ and $M = (x_1x_2) \dots (x_{2l-1}x_{2l})N$ with $N \in I^{s-l}$ for some $1 \le l \le s$. Then
$$x_n^2M = (x_nx_1)(x_2x_3) \dots (x_{2l}x_n)N \in I^{l+1+s-l} = I^{s+1}.$$
Thus, $x_n^2 \in (I^{s+1}: M).$

We proceed to prove the ``only if'' direction. Indeed, by Theorem \ref{thm.Ba}, if $x_n^2 \in (I^{s+1} : M)$ then $x_n$ must be even-connected to itself with respect to $M$. Let $x_n = p_0, p_2, \dots, p_{2l+1} = x_n$ be a shortest even-connected path between $x_n$ and itself.

Consider the case where there exists some $1 \le j \le 2l$ such that $p_j = x_n$. If $j$ is odd then $x_n = p_0, \dots, p_j = x_n$ is a shorter even-connected path between $x_n$ and itself, a contradiction. If $j$ is even then $x_n = p_j, \dots, p_{2l+1} = x_n$ is also a shorter even-connected path between $x_n$ and itself, a contradiction. Thus, we may assume that $x_n$ does not appear in the path $p_0, \dots, p_{2l+1}$ except at its endpoints.

If the path $p_0, \dots, p_{2l+1}$ is not simple, say for $1 \le i < j \le 2l$ we have $p_i = p_j$ (and we choose such $i$ and $j$ so that $j-i$ is minimal), then $p_i, \dots, p_j$ is a simple closed path lying on $C_n$. This can only occur if this simple path is in fact $C_n$, which then violates our assumption about the appearance of $x_n$ in the path $p_0, \dots, p_{2l+1}$. Therefore, $x_n = p_0, \dots, p_{2l+1} = x_n$ is a simple closed path on $C_n$. It follows that $x_n = p_0, \dots, p_{2l+1} = x_n$ is $C_n$. This, in particular, implies that $n = 2l+1$ is odd, and by re-indexing if necessary, we may assume that $p_i = x_i$ for all $i = 1, \dots, n$. Moreover, by the definition of even-connected path, we have that
$$M = (p_1p_2) \dots (p_{2l-1}p_{2l}) N = (x_1x_2) \dots (x_{2l-1}x_{2l}) N,$$
where $N$ is the product of $s-l$ edges in $C_n$ (whence, $N \in I^{s-l}$).

The last statement of the theorem follows from Theorem \ref{thm.Ba} and the following observation: for $j$ odd, $p_0, \dots, p_j$ is an even-connected path between $x_n$ and $x_j$; and for $j$ even, $p_j, \dots, p_{2l+1}$ is an even-connected path between $x_j$ and $x_n$.
\end{proof}

We are now ready to prove our next main result.

\begin{theorem} \label{thm.cycle}
Let $C_n$ be an $n$-cycle. Let $I = I(C_n)$ and let $\nu = \lfloor \frac{n}{3} \rfloor$ denote the induced matching number of $C_n$. Then
$$\reg(I) = \left\{ \begin{array}{rcll} \nu + 1 & \text{if} & n \equiv 0,1 & (\text{mod } 3) \\
\nu + 2 & \text{if} & n \equiv 2 & (\text{mod } 3), \end{array} \right.$$
and for all $s \ge 2$, we have
$$\reg(I^s) = 2s + \nu -1.$$
\end{theorem}

\begin{proof} The first statement follows from \cite[Theorem 7.6.28]{J}. We shall now prove the second statement of the theorem. In light of Theorem \ref{T1}, it suffices to show that
$$\reg(I^s) \le 2s + \nu -1.$$
By applying \cite[Theorem 5.2]{Ba} and using induction, it is enough to prove that
\begin{align}
\reg(I^{s+1}:M) \le \nu + 1 \label{eq.cycle1}
\end{align}
for any $s \ge 1$ and any minimal generator $M$ of $I^s$.

By Theorem \ref{thm.Ba}, $I^{s+1}:M$ is generated in degree 2, and its generators are of the form $uv$, where either $\{u,v\}$ is an edge in $C_n$, or $u$ and $v$ are even-connected with respect to $M$. Observe that if $x_n^2$ is a generator of $I^{s+1}:M$ then by Lemma \ref{lem.square}, we get that $n$ is odd and $x_nx_j$ is a generator of $I^{s+1}:M$ for all $j = 1, \dots, n$.
In this case, in polarizing $I^{s+1}:M$, we replace the generator $x_n^2$ by $x_ny_n$, where $y_n$ is a new variable. Thus, if we denote by $J$ the polarization of $I^{s+1}:M$ then $J$ has the form
$$J = I(G) + (x_{i_1}y_{i_1}, \dots, x_{i_t}y_{i_t}),$$
where $G$ is a graph over the vertices $\{x_1, \dots, x_n\}$, $y_{i_1}, \dots, y_{i_t}$ are new variables, and $x_{i_1}^2, \dots, x_{i_t}^2$ are all non-squarefree minimal generators of $I^{s+1}:M$. Note that polarization does not change the regularity, and we have
$$\reg(J) = \reg(I^{s+1}:M).$$
Note also that since $I^{s+1}:M$ have all edges of $C_n$ as minimal generators, $G$ has $C_n$ as a Hamiltonian cycle.

Consider the case that $I^{s+1}:M$ indeed has non-squarefree minimal generators (i.e., $t \not= 0$). For each $j = 0, \dots, t$, let $H_j$ be the graph whose edge ideal is $I(G) + (x_{i_1}y_{i_1}, \dots, x_{i_j}y_{i_j})$. Then, $H_0 = G$ and $J = I(H_t)$.

By Lemma \ref{lem.square} (and following our observation above), $\{x_{i_j},x_l\}$ is an edge in $G$ for any $j=1, \dots, t$ and any $l = 1, \dots, n$. This implies that the induced subgraph $H_j \setminus N_{H_j}[x_{i_j}]$ of $H_j$ consists of isolated vertices $\{y_{i_1}, \dots, y_{i_{j-1}}\}$. It follows that $\reg(H_j \setminus N_{H_j}[x_{i_j}]) = 0$, and by \cite[Lemma 2.10]{DHS}, we have
$$\reg(H_j) = \reg(H_j \setminus x_{i_j}).$$
However, $y_{i_j}$ is an isolated vertex in $H_j \setminus x_{i_j}$ and $H_j \setminus \{x_{i_j}, y_{i_j}\}$ is an induced subgraph of $H_j \setminus y_{i_j} = H_{j-1}$, and so we get
\begin{align}
\reg(H_j) & = \reg(H_j \setminus x_{i_j}) = \reg(H_j \setminus \{x_{i_j}, y_{i_j}\}) \nonumber \\
& \le \reg(H_j \setminus y_{i_j}) = \reg(H_{j-1}). \label{eq.cycle2}
\end{align}
Noting that $H_{j-1}$ is an induced subgraph of $H_j$, and by Theorem \ref{thm.induction}, this implies that $\reg(H_{j-1}) \le \reg(H_j)$. Therefore, coupled with (\ref{eq.cycle2}), we obtain
$$\reg(H_j) = \reg(H_{j-1}) \text{ for all } j = 1, \dots, t.$$
In particular, it follows that
$$\reg(J) = \reg(H_t) = \reg(H_0) = \reg(G).$$

To prove (\ref{eq.cycle1}), it now remains to show that $\reg(G) \le \nu + 1$. Indeed, if $n=3$ or $n=4$ then since $G$ contains a Hamiltonian path, by Theorem \ref{thm.Hamiltonpath}, we have
$$\reg(G) \le \BL \dfrac{n+1}{3} \BR + 1 = \BL \dfrac{n}{3} \BR + 1 = \nu + 1.$$
Otherwise, if $n \ge 5$ then without loss of generality we may assume that $M = (x_2x_3)M'$, where $M'$ is a generator for $I^{s-1}$, and it can be observed in this case that $x_1$ and $x_4$ are even-connected with respect to $M$. Thus, $G$ contains the edge $\{x_1,x_4\}$. It now follows from Theorem \ref{thm.Hamiltoncycle} that
$$\reg(G) \le \BL \dfrac{n}{3} \BR + 1 = \nu + 1.$$
The theorem is proved.
\end{proof}

\begin{remark} When $n \equiv 2 \ (\text{mod } 3)$ in Theorem \ref{thm.cycle}, we have
$\reg(I(C_n)^s) = 2s + \nu -1$ if and only if $s \ge 2$. Thus, in this case, $s_0 = 2 > 1$.
\end{remark}

We conclude the paper by raising the following question. This question is inspired by previous work of Herzog, Hibi and Zheng \cite{HHZ}, of Nevo and Peeva \cite{NP}, and by our main results, Theorems \ref{T2} and \ref{thm.cycle}.

\begin{question} Let $G$ be a graph with edge ideal $I = I(G)$. Let $\nu(G)$ denote the induced matching number of $G$. For which graph $G$ are the following true?
\begin{enumerate}
\item $b = \nu(G)-1$, i.e., $\reg(I^s) = 2s + \nu(G)-1$ for all $s \gg 0$.
\item $s_0 \le \reg(G) - 1$, i.e., $\reg(I^s) = 2s + b$ for all $s \ge \reg(G)-1$.
\end{enumerate}
\end{question}


\end{document}